\documentclass[11pt,a4paper]{amsart}
\usepackage[utf8]{inputenc}

\usepackage{a4wide}
\usepackage{amsmath}
\usepackage{amsthm}
\usepackage{mathrsfs}
\usepackage{wasysym}
\usepackage{enumitem}
\usepackage{cite}
\usepackage[footnotesize]{caption}
\usepackage{color}   
\usepackage{wrapfig}
\usepackage{kmath,kerkis}

\usepackage{comment}
\usepackage{amsfonts}
\usepackage{amssymb}
\usepackage{graphicx}
\usepackage[usenames,dvipsnames]{xcolor}

\newcommand{\dd}{\,d}       % used in integrals
\newcommand{\cH}{\mathcal{H}}
\newcommand{\R}{\mathbb{R}}
\newcommand{\B}{\mathbb{B}}

\newcommand{\cN}{\mathcal{N}}
\newcommand{\BMO}{\text{{\rm BMO}}}
\newcommand{\ale}{\lesssim}
\newcommand{\eps}{\varepsilon}
\newcommand{\vp}{\varphi}
\newcommand{\pl}{\partial}

\newcommand{\dv}{\operatorname{div}}
\newcommand{\ps}[2]{\left\langle #1, #2 \right\rangle}      % for the inner product
\newcommand{\wOmega}{\widetilde{\Omega}}
\newcommand{\tu}{\tilde{u}}

\newcommand{\SO}{\text{{\rm SO}}}

\newcommand{\Id}{\operatorname{Id}}

\def\barroman#1{\sbox0{#1}\dimen0=\dimexpr\wd0+1pt\relax
  \makebox[\dimen0]{\rlap{\vrule width\dimen0 height 0.06ex depth 0.06ex}%
    \rlap{\vrule width\dimen0 height\dimexpr\ht0+0.03ex\relax 
            depth\dimexpr-\ht0+0.09ex\relax}%
    \kern.5pt#1\kern.5pt}}

\def\Xint#1{\mathchoice
   {\XXint\displaystyle\textstyle{#1}}%
   {\XXint\textstyle\scriptstyle{#1}}%
   {\XXint\scriptstyle\scriptscriptstyle{#1}}%
   {\XXint\scriptscriptstyle\scriptscriptstyle{#1}}%
   \!\int}
\def\XXint#1#2#3{{\setbox0=\hbox{$#1{#2#3}{\int}$}
     \vcenter{\hbox{$#2#3$}}\kern-.5\wd0}}

\def\dashint{\Xint-}

\def\mvint_#1{\mathchoice
          {\mathop{\vrule width 6pt height 3 pt depth -2.5pt
            \kern -9.2pt \intop}\nolimits_{\kern -3pt #1}}%
             {\mathop{\vrule width 5pt height 3 pt depth -2.6pt
                        \kern -6pt \intop}\nolimits_{#1}}%
              {\mathop{\vrule width 5pt height 3 pt depth -2.6pt
                   \kern -6pt \intop}\nolimits_{#1}}%
              {\mathop{\vrule width 5pt height 3 pt depth -2.6pt
                      \kern -6pt \intop}\nolimits_{#1}}}
\newcommand{\omittedtext}[1]{\textcolor{magenta!50!black}{tu pominiety fragment}}

\numberwithin{equation}{section}

\newtheorem{theorem}{Theorem}[section]

\newtheorem{corollary}[theorem]{Corollary}
\newtheorem{lemma}[theorem]{Lemma}

\theoremstyle{definition}
\newtheorem{definition}[theorem]{Definition}
\newtheorem{remark}[theorem]{Remark}
\newtheorem{question}[theorem]{Question}

\title%
[$H$-systems and $n$-harmonic maps with $n/2$ square integrable derivatives]%
{%
Regularity for solutions of $H$-systems and $n$-harmonic maps\\
with $n/2$ square integrable derivatives
}

\date{\today}

\begin{document}

\author{Micha{\l} Mi\'{s}kiewicz}
\address[M.~Mi\'{s}kiewicz]{
  \newline \indent
  Institute of~Mathematics, 
  Polish Academy of~Sciences 
  \newline \indent
  Śniadeckich~8, 
  00-656 Warsaw, Poland 
  \newline \indent
  Institute of~Mathematics,
%  \newline{}
  University of~Warsaw
  \newline \indent
  Banacha~2,
 % \newline{}
  02-097 Warsaw, Poland}
\email{m.miskiewicz@mimuw.edu.pl}

\author{Bogdan Petraszczuk}
\address[B.~Petraszczuk]{
  \newline \indent
  Warsaw Doctoral School of Mathematics and Informatics,
\newline \indent
  Banacha~2,
 % \newline{}
  02-097 Warsaw, Poland}
\email{bogdan.petraszczuk@gmail.com}

\author{Pawe{\l} Strzelecki}
\address[P.~Strzelecki]{
  \newline \indent
  Institute of~Mathematics,
%  \newline{}
  University of~Warsaw
  \newline \indent
  Banacha~2,
 % \newline{}
  02-097 Warsaw, Poland}
\email{p.strzelecki@mimuw.edu.pl}
\maketitle

%\tableofcontents

\begin{abstract}
We study the regularity of weak solutions for two elliptic systems involving the $n$-Laplacian and a~critical nonlinearity in the right hand side: $H$-systems and $n$-har\-mo\-nic maps into compact Riemannian manifolds. Under the assumptions that the solutions belong to $W^{n/2,2}$ in an even dimension $n$, we prove their continuty.

The tools used in the proof involve Hardy spaces and BMO, and the Rivi\`{e}re--Uhlenbeck decomposition (with estimates in Morrey spaces). A prominent role is played by the Coifman--Rochberg--Weiss commutator theorem. 
\end{abstract}

\section{Introduction}

\frenchspacing

In this paper, we present two new results on the regularity of weak solutions for  two conformally invariant elliptic systems, involving the $n$-Laplacian and a~critical nonlinearity in the right hand side: $H$-systems (equations of hypersurfaces of prescribed mean curvature) and $n$-har\-mo\-nic maps into compact Riemannian manifolds. In fact, we prove that for both these problems all weak solutions $u$ of class $W^{n/2,2}$ are regular in even dimensions $n\ge 2$, cf. Theorems~\ref{Hsystems-regularity}--\ref{nharmonic-regularity} below.

In dimension $n$, by Sobolev imbedding, $W^{n/2}\subset W^{1,n}$. Thus, the extra assumption $u\in W^{n/2,2}$ is stronger than  $u\in W^{1,n}$; on the other hand, it does not trivialize the problem, as $W^{n/2,2}$ embeds neither into $L^\infty$ nor into $C^0$. Despite  various trials, up to now the counterparts of Theorems~\ref{Hsystems-regularity} and~\ref{nharmonic-regularity} under the natural assumption $u\in W^{1,n}$ are known to hold in only a~few special cases. To be specific: for the $H$-systems, one needs $H$ to be constant or to decay at infinity sufficiently fast, cf. \cite{Mou-Yang-1996b}, \cite{Wang-1999}; for the $n$-harmonic maps, one needs to assume that the target manifold $\cN$ is symmetric, e.g. a round sphere or a compact homogeneous space, see \cite{Fuchs-1993}, \cite{Strzelecki-1994},  \cite{Takeuchi-1994}, \cite{ToroW-1995}. In full generality, the problem is open. The extra assumption $u\in W^{n/2,2}$ is tailored to fill in a gap in our proofs which break down if one assumes only that $u\in W^{1,n}$; nevertheless, we do believe that they do shed new light on the general problem.

\smallskip

Before discussing what is the role of the assumption $u\in W^{n/2,2}$ in the proofs, let us state both theorems. For $n=2$, both results immediately translate to the well known results of F. H\'{e}lein's \cite{Helein-1991} and F. Bethuel's \cite{Bethuel-1992}, ascertaining the regularity of all weakly harmonic maps from planar domains into Riemannian manifolds  and of all solutions of the prescribed mean curvature equation.

\begin{theorem}
\label{Hsystems-regularity}
Let $n$ be even, let $H:\R^{n+1}\rightarrow \R$ be a bounded Lipschitz function and let $u\in W^{n/2,2}(\B^{n},\R^{n+1})$ be a weak solution of the $H$-system
\begin{equation}\label{eq1}
-{\rm{div}}(|\nabla u|^{n-2}\nabla u) = H(u)\, Ju
\end{equation}
or equivalently
\begin{equation}\label{eq2}
\int_{\B^{n}}|\nabla u |^{n-2}\nabla u \cdot \nabla \psi\, dx = \int_{\B^{n}}H(u)\, Ju \cdot \psi\, dx \qquad\mbox{for each $\psi\in C_{c}^{\infty}\big(\B^{n},\R^{n+1}\big)$,}
\end{equation}
where $Ju = \frac{\partial u}{\partial x_{1}}\times...\times\frac{\partial u}{\partial x_{n}}.$ Then $u$ is locally H\"older continuous.
\end{theorem}

Theorem~\ref{Hsystems-regularity} was announced without proof in a survey paper of A. Schikorra and the last named author: in \cite[Prop.~3.6]{Schikorra-Strzelecki-2017}, just a brief sketch of the construction of the test map and of the estimate of the right hand side was given for $n=4$. Since the whole proof is subtle and quite involved, we present it here with all necessary details. Combining the tools used in the proof of Theorem~\ref{Hsystems-regularity} with ideas of T. Rivi\`{e}re's paper \cite{Riviere-2007} and with a stability estimate of T.~Iwaniec and C.~Sbordone \cite{Iwaniec-Sbordone-1994}, we obtain the second result, Theorem~\ref{nharmonic-regularity}.

% glowne twierdzenie o przeksztalceniach n-harmonicznych

\begin{theorem}
\label{nharmonic-regularity}
Let $n$ be even, let $\cN \subseteq \R^d$ be a smooth closed submanifold, and let $u\in W^{n/2,2}(\B^{n},\cN)$ be a weakly $n$-harmonic map, i.e. a weak solution of the system
\begin{equation}\label{eq.n-harmonic}
-{\rm{div}}(|\nabla u|^{n-2}\nabla u) = |\nabla u|^{n-2} A_u(\nabla u, \nabla u),
\end{equation}
where $A_p$ is the second fundamental form of $\cN \subseteq \R^d$ at $p \in \cN$. Then $u$ is locally H\"older continuous.
\end{theorem}

The tools used in this paper involve the use of Hardy spaces and BMO, initiated by \cite{CLMS} (and developed later by too many authors to mention them all here), and the Rivi\`{e}re--Uhlenbeck decomposition \cite{Riviere-2007}. A prominent role is played by the Coifman--Rochberg--Weiss commutator theorem \cite{CRW}. 

To prove regularity, we work \textit{below the natural exponents of integrability:} we fix $\eps>0$ and show that Morrey norms $\| \nabla u \|_{L^{n-\eps,\eps}(B(a,\rho))} $ of the gradient of a solution decay like $\rho^s$ for some positive $s$; once this is done, a standard application of the Dirichlet Growth Theorem yields H\"older continuity. This idea --- to work with $\| \nabla u \|_{L^{n-\eps,\eps}(B(a,\rho))} $ instead of $\| \nabla u \|_{L^{n}(B(a,\rho))}$ --- is inspired by Iwaniec~\cite{Iwaniec-1992} and Iwaniec--Sbordone \cite{Iwaniec-Sbordone-1994}, and was used earlier in \cite{Strzelecki-2003a,Kolasinski-2010}. In particular, \textit{stability estimates} for the Hodge decomposition serve as an important tool in our construction of test functions. Both for \eqref{eq1} and \eqref{eq.n-harmonic}, the test functions $\psi$ are provided by the gradient parts of the Hodge decomposition of $F=|\nabla \tilde u|^{-\eps} \cdot P^\top \nabla\tilde u$, where $\tilde u$ denotes a cut-off of the solution and $P$ is a map into $\operatorname{SO}(m)$ satisfying $\nabla P\in L^{n}$ (for the $H$-systems, we simply work with $P\equiv\text{Id}$). Due to \cite{Iwaniec-1992,Iwaniec-Sbordone-1994} one knows that the divergence free parts of these Hodge decompositions are appropriately small, and $F$ is close to a gradient vector field.

To understand how the crucial assumption $u\in W^{n/2,2}$ is used, it is good to look at the case of the $H$-system \eqref{eq1}. In a sense, we try to see how different is the general right hand side of \eqref{eq1} from the one for \emph{constant} $H$. Namely, we split the right hand side into two terms; one of them is `good' and has a Jacobian structure, like if $H$ indeed were a constant: it poses no problems; one could handle it in numerous ways, also with a simpler test map (e.g., given by a truncated solution). The other one is `bad': the Jacobian structure is lost. The splitting of the right hand side is, roughly speaking (modulo cut-off and other technical details), based on the Hodge decomposition 
\[
b \, du^1 = d\alpha + \delta\beta\, , \qquad b=H(u).
\]
Here, the factor $b$ is not just bounded, but has vanishing mean oscillation due to the imbedding $W^{1,n}\subset \text{BMO}$. Thus, by the Coifman--Rochberg--Weiss commutator theorem, $\|\delta\beta\|_{L^n} \ale \|b\|_{\BMO}\, \|du^1\|_{L^n}$, so that, luckily, the `bad' part of the right hand side corresponding to $\delta\beta$ is small when compared to $\|\nabla u\|_{L^n}$. Just \textit{how} small it is depends on the dimension: for $n=2$, the natural assumption $u\in W^{1,n}=W^{1,2}=W^{n/2,2}$ is sufficient to match the estimates of the left and right hand side of the system. For $n>2$ things break down \emph{unless} we assume $u\in W^{n/2,2}$; the  higher order derivatives provide a means to handle the terms resulting from $|\delta\beta| \cdot |\nabla u|^{n-1}$ in the right hand side with the help of Gagliardo--Nirenberg inequalities in their sharp form,
\begin{equation}
\label{GN-intro}
\|\nabla u\,\|^{n}_{L^n}  \ale \|\nabla^{n/2} u\, \|_{L^2}^2 \|u\, \|_{\BMO}^{n-2}\, ,
\end{equation}
cf. \cite{Meyer-Riviere-2003}, \cite{Strzelecki-2006} and \cite{Riviere-Strzelecki-2005}. Now, with appropriate care,  $\|\nabla^{n/2} u\, \|_{L^2}^2$ becomes a small factor in front of the `bad' term, and the inequalities $\|b\|_{\BMO}+\|u\, \|_{\BMO} \ale \|\nabla u\|_{L^{n-\eps,\eps}}$ allow to match the estimates of the left and right hand sides. An inspection of the exponents in \eqref{GN-intro} shows that $n/2$ derivatives in $L^2$ is just what you need to close the argument.

For the $n$-harmonic maps the technical details are somewhat different, but the gist of the matter is the same. To start with, the equation does not have Jacobian structure and hence one has to employ the Rivi\`{e}re--Uhlenbeck decomposition in order to transform it into an equivalent equation with a counterpart of Jacobian structure. The remaining part of the proof resembles the one for $H$-systems, and in fact it is slightly simpler, as no decomposition is necessary. The right-hand side of the transformed equation involves a divergence-free factor of the form $A \nabla B$, so by the Coifman--Rochberg--Weiss commutator theorem, its $L^n$ norm can be bounded as $\| A \nabla B \|_{L^n} \ale \| A \|_\BMO \| \nabla B \|_{L^n}$. This upgrade from the naive $\| A \|_\infty$ estimate to $\| A \|_\BMO$ is the crucial source of decay that we establish for $\| \nabla u \|_{L^{n-\eps,\eps}}$. 

For more in-depth outlines of both arguments, we refer to Sections \ref{ch:H-systems} and \ref{ch:n-harmonic}. 

Let us mention that Kolasi\'{n}ski \cite{Kolasinski-2010} proves a weaker variant of Theorem~\ref{Hsystems-regularity}, assuming that $u\in W^{n-1,n'}$ (in his proof, there is no splitting of the right hand side into the `good' and `bad' parts as discussed above; therefore, a larger number of higher order derivatives is needed to match the estimates of the left and right hand sides). Schikorra \cite{Schikorra-2010} proves the regularity of all weak solutions with $\nabla u \in L^n {\text{log}}^{n-1-\eps} L$ for some $\eps>0$.

\smallskip

Finally, we state three related questions.

\begin{question}
Do Theorem~\ref{Hsystems-regularity} and Theorem~\ref{nharmonic-regularity} hold for odd $n>2$?
\end{question}
(One would expect that a  positive answer should use a fractional variant of \eqref{GN-intro}.)

\begin{question} Can one combine the above ideas with the results of Duzaar and Mingione \cite{Duzaar-Mingione-2010} to prove both Theorem~\ref{Hsystems-regularity} and Theorem~\ref{nharmonic-regularity} under the assumption that $\nabla u$ is in the Lorentz space $L^{n,2}$? 
\end{question}
(By the sharp form of the Sobolev imbedding we have $\nabla u \in L^{n,2}$ for $u\in W^{n/2,2}$, so this would be a stronger result.)

\begin{question} Is it possible to assume only that $u\in W^{1,n}$ and use a variant of the difference quotients method to prove that $u\in W^{n/2,2}$ \emph{before} proving regularity? 
\end{question}
(It might be more appropriate to ask for derivatives of \emph{nonlinear} functions of $\nabla u$. Consider e.g. $H$-systems and $n=4$; can one prove in this toy case that if $u\in W^{1,4}$, then $|\nabla u|\nabla u \in W^{1,4/3}$? We do not see any obvious way of achieving that; the highly non-local construction of the test maps in our proofs is one of the obstacles.)

\medskip

The rest of the paper is organized as follows. In Section~2 we gather all the preliminary material which is later used in the proofs. In Section~3, we prove Theorem~\ref{Hsystems-regularity}, and in Section~4 we prove Theorem~\ref{nharmonic-regularity}. The notation throughout the paper is standard; in particular, to avoid clutter, we use $L^n,L^{n-\eps,\eps},\BMO$ etc. to denote the function spaces on $\R^n$.

\medskip\noindent
\textbf{Acknowledgement.} The research of all three authors has been supported by the NCN Grant no. 2016/21/B/ST1/03138.
Moreover, the first named author has been supported by the NCN Grant no. 2020/36/C/ST1/00050.
\section{Preliminaries}

\subsection{Hodge decomposition}

We shall rely on $L^p$ estimates for the Hodge decomposition in $\R^n$, including a stability estimate of Iwaniec and Iwaniec--Sbordone. For the original proof of the stability estimate \eqref{stab-est}, see Iwaniec, \cite[Thm.~8.1]{Iwaniec-1992}

\begin{theorem}[Hodge decomposition and stability estimates]
\label{th.Hodge1}
Let $w:\R^{n}\rightarrow\R^m$ be of class $W^{1,p}$ for some $p>1$. Let $\varepsilon \in (-1,p-1)$. Then \[G:=|\nabla w|^{\varepsilon}\nabla w\] can be written as
$
G=\nabla \alpha + \beta,
$
with $\nabla\alpha, \beta\in L^{p/(1+\varepsilon)}(\R^{n},\R^{nm})$ and $\rm{div}\ \beta = 0$ in the sense of distributions, and 
\[
\| \nabla \alpha \|_{L^{p/(1+\varepsilon)}} + \|\beta\|_{L^{p/(1+\varepsilon)}} \leq C(n,m,p) \|\nabla u\|^{1+\varepsilon}_{L^{p}}.
\]
Moreover, 
\begin{equation}
\label{stab-est}
\|\beta\|_{L^{p/(1+\varepsilon)}} 
	\leq C(n,m,p)|\varepsilon|\,  \|\nabla w\|^{1+\varepsilon}_{L^{p}}.
\end{equation}
\end{theorem}
The proof of \eqref{stab-est} in \cite{Iwaniec-1992} relies on a long technical computation; to deal with $n$-harmonic maps, it is convenient to use another, more general theorem, due -- with a much simpler proof -- to Iwaniec and Sbordone \cite{Iwaniec-Sbordone-1994}:

\begin{theorem}
\label{thm:is-stability}
Assume that $T: L^{r}(\R^{n},\R^m) \rightarrow L^{r}(\R^{n},\R^m)$ is a linear bounded operator for each $r\in [r_1,r_2]$. 
For
\[
\frac r{r_2}-1 \le \eps \le \frac r{r_1}-1 
\]
consider
\[ S^{\varepsilon}: L^r(\R^n,\R^m)\to L^{r/(1+\eps)}(\R^n,\R^m),  \qquad S^{\varepsilon}(f) :=\Big(\frac{|f|}{\|f\|_{L^r}}\Big)^{\varepsilon}f.
\]
Then 
\[ \|TS^{\varepsilon}(f) - S^{\varepsilon}(Tf)\|_{L^{r/(1+\eps)}} \leq C_r |\varepsilon| \, \|f\|_{L^r} \]
for each $f\in L^{r}.$
\end{theorem}
In Section~4.4, we use this result combined with the Coifman--Rochberg--Weiss commutator estimate to deal with a test map produced in a non-local way, by the Hodge decomposition of $|\nabla \tu|^{-\eps}P^\top \nabla \tu$, with $P$ being a matrix--valued map of class $W^{1,n}$, cf. Theorem~\ref{th:uhlenbeck}. A word of warning: in Theorems~\ref{th.Hodge1}--\ref{thm:is-stability} $\eps$ \emph{can be negative}; in Sections~3--4 we use both results with $\eps$ replaced by $-\eps$, to stick to a standard analyst's habit.

\subsection{Morrey spaces. Dirichlet Growth Theorem} We also use Morrey spaces; actually, the H\"{o}lder continuity of solutions in Theorems~\ref{Hsystems-regularity}--\ref{nharmonic-regularity} is obtained upon an application of Dirichlet Growth Theorem. 

\begin{definition} Let $\Omega \subset \R^{n}$ be open. The Morrey space $L^{p,s}(\Omega)$ is the subspace of $L^{p}(\Omega)$ consisting of all functions $f$ for which the norm 
\[ 
\|f\|_{L^{p,s}(\Omega)} 
:= \sup_{x_{0}\in \R^n, \, r > 0}
\biggl(\frac{1}{r^{s}}\int_{B(x_0,r)\cap \Omega}|f(y)|^{p}dy\biggr)^{1/p}
\]
is finite. 
\end{definition}

\begin{theorem}[Dirichlet Growth Theorem]
\label{th-DGT}
Let $s>0$ and $p \ge 1$. Assume that $u \in W^{1,p}(B(x_0,2R))$ satisfies
\[ 
r^{p - n}\int_{B(z,r)}|\nabla u|^p dx \leq a\, r^s
\]
for every $z \in B(x_0,R)$ and all radii $0<r<R$. Then $u$ has a representative satisfying 
% for a.e. $x_1,x_2 \in B(x_0,R)$ we have
\[
|u(x_1) - u(x_2)| \leq C_0(n,s) a^{1/p}|x_1 - x_2|^{s/p}
\]
for all $x_1,x_2 \in B(x_0,R)$. 
\end{theorem}

\subsection{BMO, Hardy space estimates for Jacobians, and commutators}

\begin{definition}
A locally
integrable function $f$ belongs to the space of functions of
{\em bounded mean oscillation\/}, ${\BMO}(\R^n)$, if and only if 
\[
\|u\|_{\BMO} :=
\sup \left(\; \mvint_B |u(x)-u_B|\, dx \right) < +\infty,
\]
the supremum being taken over all balls $B$ in $\R^n$.	
\end{definition}

\begin{definition}
Let $\mathcal{F}$ be a set of all $\phi \in C_{c}^{\infty}(B(0,1))$ such that 
\[ 
\sup_{B(0,1)}|\nabla \phi| \leq 1. 
\]
For $\phi \in \mathcal{F}$ set $\phi_{\varepsilon}(x) =\varepsilon^{-n}\phi(x/\eps)$, $\eps>0$. We say that an integrable function $u: \R^{n} \rightarrow \R$ is in the Hardy space $\mathcal{H}^{1}(\R^{n})$ if and only if the maximal function 
\[  
Mu(x) = \sup_{\phi \in \mathcal{F}}\sup_{\eps>0}|(\phi \ast u)(x)| 
\]
belongs to $L^{1}.$ The norm of $u \in \mathcal{H}^{1}$ is defined as
\[
\|u\|_{\mathcal{H}^{1}(\R^{n})} := \|Mu\|_{L^{1}(\R^{n})}.   
\]
\end{definition}

Several estimates in this paper rely on the result of C.~Fefferman and E.~Stein \cite{Fefferman-Stein-1972}. 

\begin{theorem}
\label{th:h1bmo}
$\BMO(\R^{n})$ is the dual space to $\mathcal{H}^{1}(\R^{n})$; whenever $b\in\BMO(\R^n)$ and $h\in\mathcal{H}^{1}(\R^{n})$ are such that $b \cdot h$ is integrable, then
\[
\int_{\R^n} b \cdot h \, dx \ale \|b\|_{\BMO} \|h\|_{\mathcal{H}^{1}}\, .
\]
\end{theorem}

To work with Theorem~\ref{th:h1bmo}, we need an estimate for the Hardy space norm of a Jacobian determinant, due to Coifman, Lions, Meyer and Semmes \cite{CLMS}.

\begin{theorem}[$\mathcal{H}^1$ estimates for Jacobians]
\label{our.clms}
Let $u=(u^{1},\ldots,u^{n})\in W^{1,1}_{\text{loc}}(\R^{n},\R^{n})$ be such that 
\[
\nabla u^i \in L^{p_i}(\R^n), \qquad p_i>1, \qquad \sum_{i=1}^n \frac 1{p_i} = 1.
\]
Then $\det Du \in \cH^1(\R^{n})$
and
\begin{equation}
\label{det-in-h1}
    \| \det Du\|_{\cH^1(\R^{n})} \leq C\prod_{i=1}^{n}\|\nabla u^{i}\|_{L^{p_i}(\R^{n})}
\end{equation}
for some constant $C=C(n,p_1,\ldots, p_n)$.
\end{theorem}
Inequality \eqref{det-in-h1} was not explicitly stated in \cite{CLMS} but follows from the proof presented there. If $p_i$ lie in some compact interval $[\underline{p},\overline{p}] \subseteq (1,\infty)$, then the constant $C$ in \eqref{det-in-h1} depends in fact only on $n, \underline{p},\overline{p}$.

Another ingredient of the proof, as already mentioned in the Introduction, is provided by the following theorem \cite{CRW}.

\begin{theorem}[Coifman--Rochberg--Weiss commutator estimate]
\label{th:CRW-commutator}
Let $T:L^{p}(\R^n)\rightarrow L^{p}(\R^n)$, $p>1$, be a Calder\'{o}n--Zygmund singular integral operator and let $b \in \BMO(\R^n)$. Then, the commutator $[b,T]$ defined as
\[ 
[b,T](f) := b \cdot T(f)-T(b \cdot f)
\]
is bounded on $L^{p}(\R^n)$ and 
\[\|[b,T](f)\|_{L^p} \ale\|b\|_{\BMO}\, \|f\|_{L^p},\]
with a constant that depends on $p$ and $T$.
\end{theorem}

We also employ estimates of the $\BMO$ norms for the cut-off of a function $u$ in terms of the Morrey space norms of $\nabla u$. For this, we need the following.

\begin{lemma}
\label{lem:BMO-est-NEW}
Let $u$ be a function defined on the ball $B(0,\rho)$, and let $u_{B(0,\rho)}$ be its average. Choose a cut-off function $\zeta \in C_c^\infty(B(0,\rho))$ satisfying $|\nabla \zeta| \le c_0/\rho$ and define 
$$
\tilde{u} := \zeta \cdot (u-u_{B(0,\rho)}).
$$
Then 
$$
\| \nabla \tilde{u} \|_{L^{p,n-p}(\R^n)} \ale \| \nabla u \|_{L^{p,n-p}(B(0,\rho))} 
\quad \text{for all} \quad \frac n2 \le p \le n,
$$
with a constant that depends only on $n$ and on $c_0$.
\end{lemma}

\begin{proof}
Choose $\frac n2 \le p \le n$ and assume $\nabla u \in L^{p,n-p}(B(0,\rho))$. By triangle inequality and the assumption $|\nabla \zeta| \ale 1/\rho$ we have
\begin{align*}
\| \nabla \tilde{u} \|_{L^{p,n-p}(\R^n)} 
& \le \| \zeta \cdot \nabla u \|_{L^{p,n-p}(\R^n)}  + \| \nabla \zeta \cdot (u-u_{B(0,\rho)}) \|_{L^{p,n-p}(\R^n)} \\
& \ale \| \nabla u \|_{L^{p,n-p}(B(0,\rho))}  + \rho^{-1} \| u-u_{B(0,\rho)} \|_{L^{p,n-p}(B(0,\rho))},
\end{align*}
and thus we are left with bounding the second term. 

\smallskip

Let us fix a ball $B(z,r) \subset \R^n$. Applying H\"older's inequality and enlarging the domain of integration, we observe
\begin{align*}
\left( r^{p-n} \int_{B(z,r) \cap B(0,\rho)} |u-u_{B(0,\rho)}|^p \dd x \right)^{1/p}
& \ale \left( \int_{B(z,r) \cap B(0,\rho)} |u-u_{B(0,\rho)}|^n \dd x \right)^{1/n} \\
& \le \left( \int_{B(0,\rho)} |u-u_{B(0,\rho)}|^n \dd x \right)^{1/n}. 
\end{align*}
Thanks to the assumption $p \ge n/2$ we have $p^* \ge n$, which means that the Poincar\'{e}--Sobolev inequality is applicable with exponents $n$ and $p$. Finally, we obtain
\[
\left( \int_{B(0,\rho)} |u-u_{B(0,\rho)}|^n \dd x \right)^{1/n} 
\ale \rho \left( \rho^{p-n} \int_{B(0,\rho)} |\nabla u|^p \dd x \right)^{1/p} 
\ale \rho \| \nabla u \|_{L^{p,n-p}(B(0,\rho))},
\]
as required. 
\end{proof}

An application of Poincar\'e's inequality implies the desired $\BMO$ estimate: 

\begin{corollary}
\label{cor:BMO-est}
Under the notation of Lemma \ref{lem:BMO-est-NEW}, 
$$
\| \tilde{u} \|_{\BMO(\R^n)} \ale \| \nabla u \|_{L^{p,n-p}(B(0,\rho))}.
$$
In particular, this estimates applies for the space $L^n = L^{n,0}$. 
\end{corollary}

Finally, we need a variant of Rivi\'{e}re--Uhlenbeck’s decomposition for $\Omega$ of class $L^n$, with Morrey--Sobolev estimates. The following result is essentially contained in \cite{Goldstein-Zatorska-2018}. 

\begin{theorem}
\label{th:uhlenbeck}
Assume $1 < p < n/2$. Let $\Omega \colon \B^n \to \R^{d \times d} \otimes \R^n$ be an antisymmetric matrix of vector fields on $\B^n$. 
There exists $\eps_0(n,d) > 0$ such that if $\Omega$ satisfies the smallness condition 
\[
\| \Omega \|_{L^{n}} < \eps_0,
\]
then there is a matrix-valued function $P \in W^{1,n}(\B^n,\operatorname{SO}(d))$ satisfying $P|_{\pl\B^n} = \operatorname{Id}$ in the sense of trace, for which 
\begin{equation}
\label{eq:Uhlenbeck}
\wOmega := P^\top \nabla P + P^\top \Omega P
\quad \text{is divergence-free in } \B^n.
\end{equation}
Moreover, $\nabla P \in L^n$ with 
\begin{align*}
\| \nabla P \|_{L^n}
& \le C(n,d) \| \Omega \|_{L^n}, \\
\| \nabla P \|_{L^{p,n-p}} 
& \le C(n,d) \| \Omega \|_{L^{2p,n-2p}}.
\end{align*}
\end{theorem}
The first difference between the above and \cite[Thm.~1.2,1.3]{Goldstein-Zatorska-2018} is the use of vector fields versus differential forms; we choose the former just for convenience. Under the isomorphism provided by the euclidean metric, $\Omega$ and $\wOmega = P^\top dP + P^\top \Omega P$ become matrices of $1$-forms, and then the divergence-free condition on $\wOmega$ is equivalent to $\wOmega = *d\xi$ for some matrix of $(n-2)$-forms $\xi$, which leads to the formulation in \cite{Goldstein-Zatorska-2018}. 

The second, more important issue is that the actual statements in \cite[Thm.~1.2,1.3]{Goldstein-Zatorska-2018} do not cover the case of $\Omega \in L^n$. For this reason, we provide a brief sketch which bridges the gap. 

\begin{proof}[Sketch of proof of Theorem \ref{th:uhlenbeck}]
For convenience, let us consider $\Omega$ as a function on the whole of $\R^n$, simply by considering its zero extension. Using a smooth convolution kernel $\vp_\eps$, we construct smooth approximations $\Omega_j := \Omega * \vp_{1/j}$. 

For these, we can apply \cite[Thm.~1.2]{Goldstein-Zatorska-2018}, obtaining $P_j,\wOmega_j$ with regularity even higher than required here. The main point is that the $W^{1,n}$-regularity of $P_j$ is uniform, as 
\[
\| \nabla P_j \|_{L^n} 
\ale \| \Omega_j \|_{L^n} 
\le \| \Omega \|_{L^n}.
\]
We may choose a subsequence for which $\nabla P_j \to \nabla P$ weakly in $L^n$, $P_j \to P$ strongly in any $L^q$ and pointwise a.e., in particular ensuring that $P$ is $\operatorname{SO}(d)$-valued. It is now easy to see that 
\[
\wOmega_j = P_j^\top \nabla P_j + P_j^\top \Omega P_j
\to P^\top \nabla P + P^\top \Omega P
=: \wOmega
\]
weakly in $L^1$, and it follows that $\wOmega$ is divergence-free. Since $\| \nabla P \|_{L^n} \le \liminf \| \nabla P_j \|_{L^n}$ by weak convergence, we obtain the desired $L^n$-estimates. 

\medskip

It remains to show the Morrey norm estimates; this will be done via the regularity lemma \cite[Lem.~5.4]{Goldstein-Zatorska-2018}. However, this lemma requires higher regularity -- which again is not quantitatively used in our context -- therefore we apply it to $\Omega_j$, not $\Omega$ directly. Note that $\| f \|_{L^{2p,n-2p}} \ale \| f \|_{L^n}$ for any function $f$, and hence the $L^{2p,n-2p}$-smallness assumptions on  $\Omega_j$ and $\nabla P_j$ in \cite[Lem.~5.4]{Goldstein-Zatorska-2018} are satisfied. Applying this lemma, we infer 
\begin{equation}
\label{eq:Morrey-for-j}    
\| \nabla P_j \|_{L^{p,n-p}} \ale \| \Omega_j \|_{L^{2p,n-2p}},
\end{equation}
which implies an analogous estimate for $P$ and $\Omega$. To see this, first note that the bound $\| \nabla P \|_{L^{p,n-p}} 
\le \liminf \| \nabla P_j \|_{L^{p,n-p}}$ follows simply by weak $L^p$ convergence $\nabla P_j \to \nabla P$. As for $\Omega$, for a fixed $r>0$, on each ball $B(x,r)$ Minkowski inequality for integrals gives
\begin{align*}
\| \Omega_j \|_{L^p(B(x,r))}
 & = \left\| \int \Omega(\cdot-y) \vp_{1/j}(y) \dd y\,  \right\|_{L^p(B(x,r))} \\
 & \le \int \| \Omega(\cdot-y) \|_{L^p(B(x,r))} \vp_{1/j}(y) \dd y \le \sup_{y} \| \Omega \|_{L^p(B(y,r))},
\end{align*}
so $\| \Omega_j \|_{L^{2p,n-2p}} \le \| \Omega \|_{L^{2p,n-2p}}$. Combining \eqref{eq:Morrey-for-j} with these two comparisons, we have 
\[
\| \nabla P \|_{L^{p,n-p}} \ale \| \Omega \|_{L^{2p,n-2p}},
\] 
as required.
\end{proof}
\section{Regularity for $H$-systems}
\label{ch:H-systems}

In this Section, we present the proof of Theorem~\ref{Hsystems-regularity}. Here is the plan of the argument. 

Because the right-hand side of \eqref{eq1} is merely integrable and $W^{1,n}$ does not embed into $L^\infty$ (nor does $W^{n/2,2}$), the standard approach of using $u$ (or its cut-off version $\tu$) as a test function in \eqref{eq2} fails. For this reason, in Section~\ref{sec:testmap},  we employ the Hodge decomposition
\[
|\nabla \tilde{u}|^{-\varepsilon}\nabla \tilde{u} 
= \nabla \phi + V\, ,
\]
for a small $\eps>0$, and use $\phi$ (or rather its cut-off version $\psi$) as a test map. The divergence-free term $V$ is small due to the stability theorem (Theorem \ref{th.Hodge1}). Moreover, since
$(n-\eps)/(1-\eps)>n$, by Sobolev imbedding $\phi$ is bounded and we have
\begin{equation*}
\text{osc}\, \phi \ale r^{(n-1)\eps/(n-\eps)} \biggl(\int_{B_{3r}}|\nabla u|^{n-\eps}\biggr)^{\frac{1-\eps}{n-\eps}};
\end{equation*}
this is crucial for the estimates of the right-hand side. In Section~\ref{sec:LHS-H}, we deal with the left hand side of \eqref{eq2} and use Theorem~\ref{th.Hodge1} to check that
\begin{equation}
\label{lhs-sketch}
\int |\nabla u|^{n-2}\nabla u \cdot \nabla \psi\, dx \ge \int_{B_r} |\nabla  u|^{n-\eps}\, dx + \mbox{small and lower order terms.}
\end{equation}
Then, in Section~\ref{H-rhs} which forms the core of the whole proof, we check that the right hand side satisfies
\begin{equation}
\label{rhs-sketch}
\biggl| \int H(u)\, Ju \cdot \psi\, dx \biggr| \le \mbox{a small multiple of } r^\eps\, \|\nabla u\|_{L^{n-\eps,\eps}(B_{3r})}^{n-\eps}.
\end{equation}
To obtain this estimate, we first note that replacing $H(u)$ by its average would give us the term $H(u)_{B_{3r}}\, Ju$ which has Jacobian structure and hence can be handled via Hardy space estimates, enabling us to exploit $\BMO$ bounds on $\psi$ (instead of $L^\infty$ bounds) by $\cH^1-\BMO$ duality.  The 'bad' part of the right-side, as alluded to in the Introduction, measures how much $(H(u)-H(u)_{B_{3r}})\, Ju$ differs from a Jacobian. To estimate it, we employ the Coifman--Rochberg--Weiss commutator theorem \emph{and} the assumption $u\in W^{n/2,2}$, cf. Lemma~\ref{use-wn22} and its application in the proof.

Finally, in Section~\ref{sec:hole}, we combine \eqref{lhs-sketch}--\eqref{rhs-sketch} and fix the value of $\eps=\eps(n)>0$ needed to perform the standard hole filling trick. This yields the existence of an $s>0$ such that
\[
\|\nabla u\|_{L^{n-\eps,\eps}(B_{r})} \ale r^s \qquad\mbox{as $r\to 0$,}
\]
which is enough to conclude the proof.

\subsection{The test map}\label{sec:testmap} 
By a density argument, \eqref{eq2} holds for all $\psi$ of class $W_0^{1,n} \cap L^\infty$. Our test map is defined via the Hodge decomposition of $|\nabla \tilde{u}|^{-\varepsilon}\nabla \tilde{u}^{k}$, where $\tilde u$ is a cut-off solution and $\eps >0$ is sufficiently small.

Fix a standard cutoff function $\zeta \in C_{c}^{\infty}(B(a,3r))$ such that $\zeta(x) = 1$ for $x\in B(a,2r)$, $\zeta(x)=0$ for $x\in\R^n\setminus B(a,3r)$ and
\[
| \nabla^j \zeta(x)| \ale r^{-j}, \qquad j=1,2,\ldots, \frac n2\, , \quad x\in B(a,3r). 
\]
Then introduce the cut-off solution 
\[
\tilde{u}(x):=\zeta(x)(u(x)-u_{B(a,3r)}), 
\]

Pick a small $\eps>0$ (to be fixed later on) and define 
\[ 
G^{k}:=|\nabla \tilde{u}|^{-\varepsilon}\nabla \tilde{u}^{k},\qquad k=1,...,n+1. 
\]
Note that $G\in L^{q}(\R^{n})$ for $1\leq q \leq \frac{n}{1-\varepsilon}$. Indeed, 
\begin{equation}
\label{eq:GinLq}
\|G\|_{q}^{q} = \int_{B(a,3r)}|\nabla \tilde{u}|^{(1-\varepsilon)q}\, dx 
\ale \int_{B(a,3r)}|\nabla u|^{(1-\varepsilon)q}\, dx\
% \begin{split}
% \|G\|_{q}^{q} &= \int_{B(a,3r)}|\nabla \tilde{u}|^{(1-\varepsilon)q}\, dx\\
% % &=\int_{B(a,2r)}\big(\nabla\zeta\cdot (u-[u]_{A})+\nabla u \cdot \zeta)^{(1-\varepsilon)q}\, dx \\
% & \ale r^{-(1-\varepsilon)q}\int_{B(a,3r)}|u-u_{B(a,3r)}|^{(1-\varepsilon)q}\, dx 
% +\int_{B(a,3r)}|\nabla u|^{(1-\varepsilon)q}\, dx \\
% & \ale \int_{B(a,3r)}|\nabla u|^{(1-\varepsilon)q}\, dx\, 
% \end{split}
\end{equation}
by Lemma \ref{lem:BMO-est-NEW}. 
% (the last line follows from a standard application of the Poincar\'{e} inequality on $B(a,3r)$). 
In what follows we restrict the attention to 
\begin{equation}
\label{restriction-epsq}
0<\varepsilon <\frac{1}{4}, \qquad q = \frac{n-\eps}{1-\varepsilon} > n.
\end{equation}
%The constant in \eqref{eq:GinLq} depends then \emph{only on $n$}. 

Use Hodge decomposition to write
\begin{equation}\label{hodge}
G^{k}=\nabla \phi^{k} + V^{k}, \
\end{equation}
where $\phi\in W^{1,q}_{\text{loc}}(\R^{n})$, $V^{k}\in L^{q}(\R^{n},\R^{n})$ has divergence zero, and without loss of generality 
\begin{equation}
\label{mean.phi}
\dashint_{B(a,3r)}\phi^{k}\ dx = 0\, .	
\end{equation}
By Theorem~\ref{th.Hodge1} and \eqref{eq:GinLq}, 
\begin{gather} 
\label{norm.phi.V}
\|\nabla\phi^{k}\|_{L^{q}(\R^{n})}+\|V^{k}\|_{L^{q}(\R^{n})} 
\ale \|G^{k}\|_{L^{q}(\R^n)} 
\ale \biggl( \int_{B(a,3r)}|\nabla u|^{(1-\varepsilon)q}\, dx \biggr)^{1/q}\, ,\\
\label{V.small}
\|V^{k}\|_{L^{q}(\R^{n})} \le C(n)\eps \, \biggl( \int_{B(a,3r)}|\nabla u|^{(1-\varepsilon)q}\, dx \biggr)^{1/q}\, .
\end{gather}
Since $q>n$ in \eqref{norm.phi.V}, $\phi^k$ is in fact bounded and, by a routine computation involving Morrey's imbedding and \eqref{mean.phi},
\begin{equation}
	\label{bound.phi}
	\|\phi\|_{L^\infty(B(a,3r))} \le C(\eps) r^\eps \biggl(\frac{1}{r^\eps}\, \int_{B(a,3r)} |\nabla u|^{n-\eps}\, dx\biggr )^{(1-\eps)/ (n-\eps)}\, .
\end{equation}
% See [Strzelecki2003], Lemma 2.2, inequality (2.12)

\begin{remark} For $\eps,q$ satisfying \eqref{restriction-epsq}, the constant in \eqref{eq:GinLq} and \eqref{norm.phi.V} depends \emph{only on $n$}. This follows from the Riesz--Thorin convexity theorem applied to the singular operators that produce the gradient part and the divergence free part of the Hodge decomposition. However, the constant in \eqref{bound.phi} \emph{does} also depend on $\eps>0$, as there is no imbedding of $W^{1,n}$ into $L^\infty$.
\end{remark}

We test \eqref{eq1} with
\[
\psi^{k}:=\zeta_0 \phi^{k}, \qquad k=1,2,\ldots, n+1,
\]
where $\zeta_0$ is \emph{another} cut-off function: $\zeta_0 \in C_c^\infty(B(a,2r))$ satisfying $\zeta_0 \equiv 1$ on $B(a,r)$. This way, $\nabla u$ and $\nabla \tu$ conveniently coincide on the support of $\psi$.

\subsection{Left hand side estimates}\label{sec:LHS-H} 
Using the definition of $\psi$ and $\phi$, we write

\begin{align*}
\int |\nabla u|^{n-2} \nabla u \nabla \psi 
& = \int |\nabla u|^{n-2} \nabla u (\zeta_0 \nabla \phi + \nabla \zeta_0 \, \phi) \\
& = \int \zeta_0 |\nabla u|^{n-2} \nabla u \cdot |\nabla u|^{-\eps} \nabla u 
- \int \zeta_0 |\nabla u|^{n-2} \nabla u \cdot V \\
& \phantom{=} + \int |\nabla u|^{n-2} \nabla u \cdot \nabla \zeta_0 \, \phi  \\
& = \barroman{I} + \barroman{II} + \barroman{III}.
\end{align*}
Note how we are free to replace $\nabla \tu$ with $\nabla u$, as these two maps agree on the support of $\zeta_0$. 

\medskip

The term $\barroman{I}$ dominates the energy, and in fact 
\[
\barroman{I} 
= \int \zeta_0 |\nabla u|^{n-\eps} 
\ge \int_{B(a,r)} |\nabla u|^{n-\eps}.
\]
The error term $\barroman{II}$ is small due to the stability estimate \eqref{V.small}: 
\begin{equation}
\label{J2-est}
	\begin{split}
	\barroman{II} 
	& \le \int_{B(a,2r)} |\nabla u|^{n-2} |\nabla u| \, |V|\, dx \\
	& \le \left( \int_{B(a,2r)} |\nabla u|^{n-\eps} \right)^{\frac{n-1}{n-\eps}} \left( \int_{B(a,2r)}  |V|^{\frac{n-\eps}{1-\eps}} \right)^{\frac{1-\eps}{n-\eps}} \\
	& \le C(n)\eps \int_{B(a,3r)} |\nabla u|^{n-\eps}\, dx\, .
	\end{split}	
\end{equation}

The term $\barroman{III}$, with $\nabla \zeta\not=0$ only on the annulus $A = B(a,2r) \setminus B(a,r)$, is a lower order boundary term. As $|\nabla\zeta|\ale r^{-1}$, a standard computation employing H\"{o}lder and Poincar\'{e} inequalities yields
\begin{equation*}
\begin{split}
|\barroman{III}|
& \ale r^{-1} \int_{A} |\nabla u|^{n-1}|\phi| \, dx  \\
& \ale r^{-1} 
\biggl(\int_{A}|\nabla u|^{n-\varepsilon}\, dx\biggr)^{(n-1)/(n-\eps)}
\biggl(\int_{B(a,2r)}
 |\phi|^{(n-\eps)/(1-\eps)}\, dx\Big)^{(1-\eps)/(n-\epsilon)} \\
& \ale \biggl(\int_{A}|\nabla u|^{n-\varepsilon}\, dx\biggr)^{(n-1)/(n-\eps)}
\biggl(\int_{B(a,2r)}
 |\nabla \phi|^{(n-\eps)/(1-\eps)}\, dx\Big)^{(1-\eps)/(n-\epsilon)}. 
\end{split}
\end{equation*}
Invoking now \eqref{norm.phi.V} for $q  = \frac{n-\eps}{1-\eps}$ and Young's inequality, we obtain
\begin{equation}
\label{J1-est}
|\barroman{III}| \le C(n) \int_{A}|\nabla u|^{n-\varepsilon}\, dx + \frac{1}{10} \int_{B(a,3r)}|\nabla u|^{n-\varepsilon}\, dx
\end{equation}

\medskip

Assuming from now on that $\eps>0$ is small so that 
\begin{equation}
\label{eps.1}
	C(n)\eps< \frac{1}{10} \qquad\mbox{in \eqref{J2-est},}
\end{equation} 
and gathering the estimates of $\barroman{I},\barroman{II},\barroman{III}$, we finally arrive at
\begin{equation}
\label{final.LHS}
\begin{split}
\int |\nabla u|^{n-2}\nabla u\nabla\psi\,  dx 	& 
\ge \int_{B(a,r)} |\nabla u|^{n-\eps}\, dx 	\\
& \quad{} - C(n) \int_{A}|\nabla u|^{n-\varepsilon}\, dx - \frac{1}{5} \int_{B(a,3r)}|\nabla u|^{n-\varepsilon}\, dx\, .
\end{split}
\end{equation}

\subsection{Right hand side estimates}
\label{H-rhs}
Here, we stick to the language of differential forms. The $i$-th  coordinate $J(u)^i$ of  $Ju = \frac{\partial u}{\partial x_{1}}\times\ldots \times\frac{\partial u}{\partial x_{n}}$ 
satisfies
\[
	J(u)^{i}dx_{1}\wedge\ldots \wedge dx_{n} 
	= u^{\ast}(\star dx^{i})
	= du^{1}\wedge \ldots \wedge du^{i-1}
	    \wedge du^{i+1}\wedge \ldots \wedge du^{n+1}\, ,
\] 
where $\star$ is the Hodge star. Thus, the right hand side of \eqref{eq2} is equal to the sum of 
\begin{equation}
\label{def.ti}
t_i=\int_{\B^{n}} H(u) \psi^i 
		du^{1}\wedge\ldots \wedge du^{i-1}
		\wedge du^{i+1}\wedge\ldots\wedge du^{n+1}\, .	
\end{equation}
We show how to deal with the term $t=t_{n+1}$; all the other ones are handled in the same way.

As we have already mentioned, the main idea of the proof is to split the right side into two kinds of terms. The terms of the first kind are harmless: they behave like if $H(u)$ were \textbf{a constant}. The remaining bad error terms are handled using the Coifman--Rochberg--Weiss commutator theorem and Gagliardo--Nirenberg inequalities in a borderline case; here the assumption $u\in W^{n/2, 2}$ enters in a decisive, crucial way. Here are the details.

To keep $u$ unchanged on the support of $\psi=\zeta\phi$, i.e. on $B(a,2r)$, and localize it in a~comparable scale, we have already introduced the cut-off function $\zeta$ with $\zeta \equiv 1 $ on $B(a,2r)$ and $\zeta \equiv 0$ off $B(a,3r)$. For brevity, write
\begin{equation}
\label{Hv}
\tu = \zeta (u - u_{B(a,3r)}), \qquad 
b = \zeta_0 (H(u)-H(u)_{B(a,2r)}), \qquad 
H_0 = H(u)_{B(a,2r)}.
\end{equation}
The right hand side term $t=t_{n+1}$ in \eqref{def.ti} becomes then
\begin{equation}
\label{def.R1R2}
\begin{split}
t &= \int_{\B^{n}}  
	 \phi^{n+1} b\,  d\tu^{1}\wedge\ldots \wedge d\tu^{n}
  + H_0 \int_{\B^{n}} \psi^{n+1} d\tu^1\wedge \ldots \wedge d\tu^n\, \\
  & = R_1 +R_2\, . 	
\end{split}
\end{equation}

\subsubsection{The harmless term $R_2$.}
Since $\tu$ is compactly supported, after one integration by parts (to apply $d$ to $\psi^{n+1}$) we use Theorem~\ref{our.clms} to obtain
\begin{equation}
\label{R2-est}
	\begin{split}
		|R_2| &\ale \|\tu^1\|_{\BMO}\,\,  \|d\psi^{n+1} \wedge d\tu^2 \wedge \ldots \wedge d\tu^n\|_{\cH^1(\R^n)}\\
		& \ale \|\tu\,\|_{\BMO}\,\,  \|\nabla \psi\|_{L^{(n-\eps)/(1-\eps)}} 						  \,\,  \|\nabla \tu\,\|_{L^{n-\eps}}^{n-1} \\
		& \ale \|\tu\,\|_{\BMO} \int_{B(a,3r)} |\nabla u|^{n-\eps}\, dx\, , 
	\end{split}	
\end{equation}
as
\[
\|\nabla \tu\|_{L^{n-\eps}}^{n-\eps}  + \|\nabla \psi\|_{L^{(n-\eps)/(1-\eps)}}^{(n-\eps)/(1-\eps)}
\ale \int_{B(a,3r)} |\nabla u|^{n-\eps}\, dx 
\]
by standard applications of Poincar\'{e} inequality and \eqref{norm.phi.V}. The constant in \eqref{R2-est} depends only on $n$ and $\|H\|_\infty$.

\subsubsection{A further splitting of $R_1$.} It turns out that $R_1$ is a sum of another harmless term and an error term. Use Hodge decomposition to write
\[
b \, d\tu^1 = d\alpha + \delta\beta
\]
and note that since $H$ is bounded, by Theorem~\ref{th.Hodge1} and the Coifman-Rochberg-Weiss commutator estimate, cf. Theorem~\ref{th:CRW-commutator} (applied to the operator $T$ which maps a vector field to the gradient component of its Hodge decomposition), we have 
\begin{equation}
\label{ab-est}
\|d\alpha\|_{L^{n-\eps}}\ale \|d\tu\,\|_{L^{n-\eps}}, \qquad \|\delta\beta\|_{L^n} \ale \|b\|_{\BMO}\, \|d\tu\,\|_{L^n}\, .	
\end{equation}
By \eqref{def.R1R2}, 
$
R_1=R_{11}+ R_{12}
$,
with
\begin{equation}
\label{def.R11R12}
\begin{split}
R_{11} & = \int   
	 \phi^{n+1} \delta\beta\wedge d\tu^{2}\wedge\ldots \wedge d\tu^{n}\, ,\\
R_{12} & = \int   
	 \phi^{n+1} d\alpha\wedge d\tu^{2}\wedge\ldots \wedge d\tu^{n}\, .\\
\end{split}
\end{equation}
We handle the harmless $R_{12}$ precisely as $R_2$ above, in \eqref{R2-est}, to obtain
\begin{equation}
\label{R12-est}
\begin{split}
	|R_{12}| & \ale \|\tu^n\|_{\BMO}\, 
					\|\nabla \phi\, \|_{L^{(n-\eps)/(1-\eps)}}\,
					\|d\alpha\|_{L^{n-\eps}} \, 
					\|\nabla \tu\, \|_{L^{n-\eps}}^{n-2} \\
			 & \ale \|\tu\, \|_{\BMO} \int_{B(a,3r)} |\nabla u|^{n-\eps}\, dx
			 		\, .
\end{split}
\end{equation}
Gathering \eqref{R2-est} and \eqref{R12-est}, and using Corollary \ref{cor:BMO-est}, we obtain 
\begin{equation}
	\label{RHS.easy}
	|R_2| +|R_{12}| 
	\ale 
		\biggl(\int_{B(a,3r)} |\nabla u|^n\, dx \biggr)^{1/n}\, 
		\int_{B(a,3r)} |\nabla u|^{n-\eps}\, dx\, . 
\end{equation}

\subsubsection{The error term and the use of higher order derivatives} We proceed now to the bad term $R_{11}$. By H\"{o}lder's inequality and \eqref{ab-est}, we have
\begin{equation}
\label{R11.first}
\begin{split}
|R_{11}|  
& \le \|\phi \|_{L^\infty} \|\delta\beta\|_{L^n} \|\nabla \tu\, \|^{n-1}_{L^n}\\
& \ale \|\phi\|_{L^\infty} \, \|b\|_{\BMO}\, \|\nabla \tu\,\|^{n}_{L^n}\, .
\end{split}
\end{equation}
Invoking the estimate \eqref{bound.phi} for $\phi$, Corollary~\ref{cor:BMO-est} for $b$ (combined with the assumption that $H$ is bounded and Lipschitz), we easily check that  \eqref{R11.first} implies 
\begin{equation}
\label{R11.second}
|R_{11}| \le C(\eps) r^\eps \|\nabla u\|_{L^{n-\eps,\eps}(B(a,3r))}^{2-\eps} \|\nabla \tu\,\|^{n}_{L^n}\, .
\end{equation}
In order to be able to combine this estimate with those of $R_2$ and $R_{12}$ given by \eqref{RHS.easy}, and the left hand side estimate \eqref{final.LHS}, we need the following.  

\begin{lemma} 
\label{use-wn22}
For $u\in W^{n/2,2}$ and $\tu=\zeta (u-u_{B(a,3r)})$ as above ($\zeta \equiv 1$ in $B(a,2r)$, $\zeta \equiv 0$ outside $B(a,3r)$, $|\nabla^j \zeta| \ale r^{-j}$ for $j = 1,\ldots,n/2$), we have
\begin{equation}
\label{comparenorms}
\begin{split}
\|\nabla \tu\,\|^{n}_{L^n}  & \ale \|\nabla^{n/2} \tu\, \|_{L^2}^2 \|\tu\, \|_{\BMO}^{n-2}\,  \\  
& \ale \Psi(a,3r)^2 \, \|\nabla u\|_{L^{n-\eps,\eps}(B(a,3r))}^{n-2}\, ,
\end{split}
\end{equation}
where
\begin{equation}
\label{Psi.3r}
\Psi(a,r) := 
    \| \nabla u \|_{L^n(B(a,r))} + \| \nabla^2 u \|_{L^{n/2}(B(a,r))} + \ldots + \| \nabla^{n/2} u \|_{L^2(B(a,r))}\, .
\end{equation}
\end{lemma}
\begin{proof}
The first line in \eqref{comparenorms} follows directly from \cite[Theorem~2]{Strzelecki-2006}. To check the second inequality, we apply Corollary~\ref{cor:BMO-est} to obtain
\[
\|\tu\, \|_{\BMO}\ale \|\nabla u\|_{L^{n-\eps,\eps}(B(a,3r))}\, .
\]
To estimate $\nabla^{n/2} \tu$, we employ the Leibniz formula for higher order derivatives of the product of two functions, $\zeta$ and $u-u_{B(a,3r)}$, and perform a routine computation involving inductive applications of Sobolev's inequality to verify that
\[
\|\nabla^{n/2} \tu \, \|_{L^2} \ale \Psi(a,3r)\, ,
\]
with $\Psi$ defined by \eqref{Psi.3r}.
\end{proof}

Invoking \eqref{RHS.easy}, \eqref{R11.second} and Lemma~\ref{use-wn22}, we obtain the final estimate of the right hand side,
\begin{equation}
	\label{final.RHS}
\begin{split}
	\biggl|
	\int \psi H(u) J(u)\, dx 
	\biggr|
	& \le C \biggl(\int_{B(a,3r)} |\nabla u|^n\, dx \biggr)^{1/n}\, 
		\int_{B(a,3r)} |\nabla u|^{n-\eps}\, dx\, 
		\\
	& \qquad{} 
		+  C(\eps) r^\eps \, \Psi(a,3r)^2\,   
		   \|\nabla u\|_{L^{n-\eps,\eps}(B(a,3r))}^{n-\eps} \, \\
	& \le C(\eps) r^\eps\, \Phi(a,3r)\, 
	      \|\nabla u\|_{L^{n-\eps,\eps}(B(a,3r))}^{n-\eps} ,
\end{split}
\end{equation}
where, for sake of brevity,
\begin{equation}
\label{small.Phi}
\Phi(a,\rho) = \Psi(a,\rho) + \Psi(a,\rho)^2\, .
\end{equation}

\subsection{Choice of $\epsilon$ and hole filling} 
% (fold)
\label{sec:hole}
\label{sub:3.4}
We now combine \eqref{final.LHS} and \eqref{final.RHS} to obtain
\[
\int_{B(a,r)} |\nabla u|^{n-\eps}\, dx 	\le 
 C(n) \int_{A}|\nabla u|^{n-\varepsilon}\, dx 
 + \frac{1}{5} \int_{B(a,3r)}|\nabla u|^{n-\varepsilon}\, dx
 + C(\eps) r^\eps\, \Phi(a,3r)\, 
 \|\nabla u\|_{L^{n-\eps,\eps}(B(a,3r))}^{n-\eps} \, .
\]
Thus, by hole filling,
\begin{equation}
\label{almost}
\int_{B(a,r)} |\nabla u|^{n-\eps}\, dx 	\le 
 \lambda_0  \int_{B(a,3r)}|\nabla u|^{n-\varepsilon}\, dx
 + C_1(\eps) r^\eps\, \Phi(a,3r)\, \|\nabla u\|_{L^{n-\eps,\eps}(B(a,3r))}^{n-\eps} \, ,	
\end{equation}
with $\lambda_0=\lambda_0(n)= \big(C(n)+\frac 15\big)/ \big(C(n)+1\big) < 1$. For this $\lambda_0$, we fix $\lambda_1,\lambda_2$ so that
\[
\lambda_0 < \lambda_1 < \lambda_2 < 1. 
\]
From \eqref{almost}, we obtain
\begin{equation}
	\label{fixing}
	\frac{1}{r^\eps}\int_{B(a,r)} |\nabla u|^{n-\eps}\, dx
	\le 3^\eps \lambda_0 \frac{1}{(3r)^\eps}\int_{B(a,3r)} 
			|\nabla u|^{n-\eps}\, dx
	+
	C_1(\eps)\, \Phi(a,3r)\, 
	\|\nabla u\|_{L^{n-\eps,\eps}(B(a,3r))}^{n-\eps} \, .
\end{equation}
Now we first fix $\eps=\eps(n)>0$ small enough to have both \eqref{eps.1} and $3^\eps\lambda_0< \lambda_1$. Then, we use absolute continuity of integral to fix $R_0=R_0(\eps)$ so that
\[
C_1(\eps) \Phi(z,R_0) < \lambda_2-\lambda_1 \qquad\mbox{for all $z\in \B^n$.}
\]
Then, whenever $3r<R_0$, we have
\[
\frac{1}{r^\eps}\int_{B(a,r)} |\nabla u|^{n-\eps}\, dx
\le \lambda_2 \|\nabla u\|_{L^{n-\eps,\eps}(B(a,3r))}^{n-\eps}\, .
\]
Noting that if $B(z,\rho)\subset B(a,r)$, then $B(z,3\rho)\subset B(a,3r)$, we write the above inequality with $a,r$ replaced by $z,\rho$ and take the supremum over all $B(z,\rho)\subset B(a,r)$ to obtain
\[
\|\nabla u\|_{L^{n-\eps,\eps}(B(a,r))}^{n-\eps} \le \lambda_2\, \|\nabla u\|_{L^{n-\eps,\eps}(B(a,3r))}^{n-\eps}\, .
\]
% subsection choice_of_epsilon_and_hole_filling (end)
Now, a standard iterative argument shows that for $s=-\log(\lambda_2)/\log 3>0$ we have
\[
\|\nabla u\|_{L^{n-\eps,\eps}(B(a,r))}^{n-\eps}
\ale (r/R)^s \|\nabla u\|_{L^{n-\eps,\eps}(B(a,R))}^{n-\eps}, \qquad 0<r<R\le R_0\, .
\]
 An application of Theorem~\ref{th-DGT} yields H\"{o}lder continuity of $u$.

\section{Regularity for $n$-harmonic maps}
\label{ch:n-harmonic}

This Section is devoted to the proof of Theorem \ref{nharmonic-regularity}; we begin with an outline. 

As the solution $u$ maps into a closed submanifold of $\R^d$, in principle one can use $u$ (with a cut-off) as a test function in the $n$-harmonic equation \eqref{eq.n-harmonic}. However, if one only uses the obvious $L^1$ bound of the right-hand side together with $\| u \|_\infty \le C$, this naive approach fails. 

Thus, we need to improve the $L^1$-estimate of the right-hand side by exploting its structure (see Section \ref{ch:n-harm-structure}). Rewriting the second fundamental form term $A(\nabla u, \nabla u)$ as $\Omega \cdot \nabla u$, where $\Omega_{ij} = \sum_l (A_{jl}^i-A_{il}^j) \nabla u^l$, we observe that $\Omega_{ij}$ is \emph{almost} a gradient field. To be precise, if one applies the Hodge decomposition to $\Omega$, the Coifman--Rochberg--Weiss commutator theorem implies that the $L^n$-norm of its divergence-free part is bounded by the $L^n$-norm of $\nabla u$ times the $\BMO$-norm of the coefficients $A_{jl}^i-A_{il}^j$. This last norm is bounded by the $\BMO$-norm of the solution $u$. 

In order to take advantage of this gain, in Section \ref{ch:n-harm-Uhlenbeck} we use Uhlenbeck's decomposition, which enables us to replace $\Omega$ with a matrix of \emph{divergence-free} vector fields $\wOmega$. This new matrix is given as $\wOmega=P^\top \nabla P + P^\top \Omega P$ for a proper choice of $\SO(d)$-valued function $P$, and leads to the~transformed, equivalent form of \eqref{eq.n-harmonic} (see \eqref{eq:transformed-n-harmonic}): 
\[
- \dv (|\nabla u|^{n-2} P^\top \nabla u) 
= |\nabla u|^{n-2} \wOmega \cdot P^\top \nabla u.
\]
The commutator argument described above still applies with $\Omega$ replaced by $\wOmega$, and yields $\| \wOmega \|_{L^n} \ale \| u \|_{\BMO} \| \nabla u \|_{L^n}$ (see Lemma \ref{lem:L1-smallness}). By Poincar\'e's inequality, the additional $\BMO$ term is small on small balls, which results in improved $L^1$-estimates for the right-hand side. When testing with a bounded function, the resulting integral is much smaller than the $W^{1,n}$-energy of $u$. 

\medskip

However, this gain comes at a price: testing the transformed equation with $u$ does not produce the $W^{1,n}$-energy on the left-hand side. To obtain it, one is forced to test e.g., with a~function $\phi$ resulting from the Hodge decomposition of $P^\top \nabla u$. In turn, this leads to a new problem: such a function does not necessarily lie in $L^\infty$, and thus it cannot be used as a test function. 

We choose to work around this issue by applying the Hodge decomposition of $|\nabla u|^{-\eps} P^\top \nabla u$ for a small value of $\eps > 0$, and testing with $\phi$ corresponding to the gradient part. Such $\phi$ is continuous by Morrey's embedding, which makes it a suitable test function. Deciding to work below the natural exponent, we aim (as in Section~3) at a decay estimate for the Morrey norm $\| \nabla u \|_{L^{n-\eps,\eps}}$. This makes the $L^n$-norm of $\nabla u$ appearing naturally on the right-hand side undesirable, and we employ our additional assumption $u \in W^{n/2,2}$ to handle it. 

The final Section \ref{ch:n-harm-finish} involves the choice of $\eps > 0$ and the standard hole-filling trick, and establishes the decay estimate for $\| \nabla u \|_{L^{n-\eps,\eps}}$. 
Since the argument given for $H$-systems applies here verbatim, we actually refer to Section \ref{sec:hole} for details.

\subsection{Structure of the equation}
\label{ch:n-harm-structure}

Let us start by rephrasing the equation in a way consistent with Uhlenbeck's decomposition (Theorem \ref{th:uhlenbeck}). To this end, we extend the second fundamental form $A_p$ to a symmetric bilinear form on the whole $\R^d$ (not just $T_p \cN$), e.g. by taking the tangential part of both arguments, and denote its components $(A_{ij})_p = A_p(e_i,e_j)$ in the standard basis of $\R^d$. Using the tubular neighborhood theorem, we also extend $A_{ij}$ to functions in $C_c^\infty(\R^d)$. In the sequel, $A_{ij}$ (without any subscript) will denote the composition of $A_{ij}$ with $u$. 

Taking advantage of the fact that $(A_{ij})_{u(x)} \perp \pl_\alpha u(x)$ for each $\alpha = 1,\ldots,n$, $i,j=1,\ldots,d$ and $x \in \B^n$ -- as $\pl_\alpha u(x)$ is tangent to $\cN$ at $u(x)$ and $(A_{ij})_{u(x)}$ is orthogonal -- we can rewrite the right hand side as follows: 
\begin{align*}
    A_u^i(\nabla u, \nabla u)    
    & = \sum_\alpha A^i(\pl_\alpha u, \pl_\alpha u) 
    = \sum_{\alpha,j,l} A^i_{jl} \pl_\alpha u^j \pl_\alpha u^l \\
    & = \sum_{\alpha,j,l} (A^i_{jl} - A^j_{il}) \pl_\alpha u^j \pl_\alpha u^l 
    = \sum_{j} \Omega_{ij} \cdot \nabla u^j, \\
\end{align*}
if we take $\Omega_{ij} := \sum_l (A^i_{jl} - A^j_{il}) \nabla u^l$. In this notation, $\Omega$ is a $d \times d$ matrix, entries of which are vector fields on $\B^n$. Notice that $\Omega$ is skew-symmetric and that $|\Omega| \ale |\nabla u|$ pointwise, with a~constant depending on $\cN$. In short, the equation for $u$ takes the form
\begin{equation}
    \label{Omega-eq}
    - \dv (|\nabla u|^{n-2} \nabla u) 
= |\nabla u|^{n-2} \Omega \cdot \nabla u.
\end{equation}
For reasons that should become clear later, it is advantageous to use a cut off function already at this stage. We alter the definition of $\Omega$ slightly, as follows. Fix a cutoff function $\zeta \in C_c^\infty(B(a,3r))$ satisfying $\zeta \equiv 1$ on $B(a,2r)$, and
\[
| \nabla^j \zeta(x)| \ale r^{-j}, \qquad j=1,2,\ldots, \frac n2\, , \quad x\in B(a,3r). 
\]
We take 
\[
\tu(x) := \zeta(x) (u(x) - u_{B(a,3r)}) + u_{B(a,3r)};
\]
note the added constant which ensures that $\tu$ and $u$ coincide in $B(a,2r)$ (not only their derivatives). Now we let 
\[
\Omega_{ij} 
:= \sum_l (A^i_{jl}(\tu) - A^j_{il}(\tu)) \nabla \tu^l.
\]
and observe that the equation \eqref{Omega-eq} still holds on $B(a,2r)$. On the other hand, now $\Omega$ vanishes outside of $B(a,3r)$, which will be useful later on. 

\medskip 

For future use, let us note the size estimates on $\Omega$. We observe that $|\Omega| \ale |\nabla \tu|$ pointwise, which implies by an application of Lemma \ref{lem:BMO-est-NEW} that 
\begin{align*}
    \| \Omega \|_{L^n(\R^n)} & \ale \| \nabla \tu \|_{L^n(\R^n)} 
    \ale \| \nabla u \|_{L^n(B(a,3r))}, \\
    \| \Omega \|_{L^{n-\eps,\eps}(\R^n)} & \ale \| \nabla \tu \|_{L^{n-\eps,\eps}(\R^n)} 
    \ale \| \nabla u \|_{L^{n-\eps,\eps}(B(a,3r))}.
\end{align*}

\subsection{Application of Uhlenbeck's decomposition}
\label{ch:n-harm-Uhlenbeck}

We may assume that the smallness assumption $\| \Omega \|_{L^n(\B^n)} < \eps_0$ from Theorem \ref{th:uhlenbeck} is satisfied; in consequence, it also holds on any smaller ball. If this were not true, we could still cover $\B^n$ by smaller balls on which the smallness condition holds, and work on them instead.

Consider the decomposition given by Theorem \ref{th:uhlenbeck} on the ball $B(a,3r)$ for the exponent $p = \frac{n-\eps}{2}$. Since the norms involved are scale-invariant, we infer that 
\[
\wOmega := P^\top \nabla P + P^\top \Omega P
\]
is a matrix of divergence-free vector fields, for some $\SO(d)$-valued function $P \in W^{1,n}$ satisfying $P = \Id$ on $\pl B(a,3r)$. We may consider $P$ as extended beyond this ball by $P \equiv \Id$, which results in $\wOmega = \Omega = 0$ outside $B(a,3r)$. 

Moreover, estimates in $L^n$ and $L^{p,n-p}$ are available for $P$: 
\begin{align*}
\| \nabla P \|_{L^n(B(a,3r))}
& \le C(n,d) \| \Omega \|_{L^n(B(a,3r))}, \\
\| \nabla P \|_{L^{p,n-p}(B(a,3r))} 
& \le C(n,d) \| \Omega \|_{L^{n-\eps,\eps}(B(a,3r))}. 
\end{align*}
Due to the zero extension, we can replace $B(a,3r)$ above with $\R^n$; together with the previous estimates on $\Omega$, these give us 
\begin{equation}
\label{eq:P-est}
\begin{split}
\| \nabla P \|_{L^n(\R^n)}
& \ale \| \nabla \tu \|_{L^n(\R^n)}, \\
\| \nabla P \|_{L^{p,n-p}(\R^n)} 
& \ale \| \nabla u \|_{L^{n-\eps,\eps}(B(a,3r))}. 
% \nonumber
\end{split}    
\end{equation}
Note that the norm $\| P \|_\BMO$ is controlled by $\| \nabla P \|_{L^{p,n-p}(\R^n)}$ above.

\medskip

On $B(a,2r)$, a direct computation yields the following transformed equation
\[
- \dv (|\nabla u|^{n-2} P^\top \nabla u) 
= |\nabla u|^{n-2} \left( P^\top \Omega - \nabla P^\top \right) \cdot \nabla u.
\]
Taking into account $P \, P^\top = \Id$, we have $\nabla P^\top = - P^\top \, \nabla P \, P^\top$, and thus 
\begin{equation}
\label{eq:transformed-n-harmonic}
- \dv (|\nabla u|^{n-2} P^\top \nabla u) 
= |\nabla u|^{n-2} \wOmega \cdot P^\top \nabla u
\quad \text{in } B(a,2r).
\end{equation}
This form of the equation is bit more involved than the original, but the crucial piece of structure here is that $\dv \wOmega = 0$. 

\subsection{The test map}

Similarly to the case of $H$-systems, we will define the test map by applying the Hodge decomposition to $|\nabla \tilde{u}|^{-\eps} P^\top \nabla \tu$, where $\tu$ is the cut-off solution from the previous subsection and $\eps > 0$ is small. 

Recall that $\tu = \zeta (u - u_{B(a,3r)}) + u_{B(a,3r)}$, where $\zeta \in C_c^\infty(B(a,3r))$ is a cut-off function satisfying $\zeta \equiv 1$ on $B(a,2r)$. With small $\eps > 0$ to be chosen later, we define 
\[
G^k := |\nabla \tu|^{-\eps} (P^\top \nabla \tu)^k, 
\quad k=1,\ldots,d.
\]
The function $G$ is in $L^q(\R^n)$ for all $1 \le q \le \frac{n}{1-\eps}$. However, we will use the value $q := \frac{n-\eps}{1-\eps}$, which lies strictly between $n$ and $\frac{n}{1-\eps}$. We have 
\[
\| G \|_{L^q(\R^n)}^q 
\le \int_{\R^n} |\nabla \tu|^{n-\eps} \dd x
\ale \int_{B(a,3r)} |\nabla u|^{n-\eps} \dd x
\]
by Lemma \ref{lem:BMO-est-NEW}. Then by Hodge decomposition (Theorem \ref{th.Hodge1}), we have 
\[
G^k = \nabla \phi^k + V^k 
\quad \text{for each } k,
\]
where $\phi^k \in W_{\text{loc}}^{1,q}(\R^n)$, $V^k \in L^q(\R^n,\R^n)$ is divergence-free and
\[
\|\nabla\phi^{k}\|_{L^{q}(\R^{n})}+\|V^{k}\|_{L^{q}(\R^{n})} 
\ale \|G^{k}\|_{L^{q}(\R^n)} 
\ale \left( \int_{B(a,3r)}|\nabla u|^{n-\eps} \dd x \right)^{(1-\eps)/(n-\eps)}.
\]

Moreover, without loss of generality we can have 
\[
\dashint_{B(a,3r)} \phi^k = 0.
\]
It will be crucial that $q = \frac{n-\eps}{1-\eps}$ is larger than $n$, since then an application of Morrey's inequality gives us 
\[
\| \phi \|_{L^\infty(B(a,3r))} 
\ale C(\eps) r^{\eps} \left( r^{-\eps} \int_{B(a,3r)} |\nabla u|^{n-\eps} \right)^{(1-\eps)/(n-\eps)}. 
\]
Introducing another cut-off function $\zeta_0 \in C_c^\infty(B(a,2r))$ satisfying $\zeta_0 \equiv 1$ on $B(a,r)$, we can finally choose $\psi(x) := \zeta_0(x) \phi(x)$ as the test function. Note that it is supported in the ball $B(a,2r)$ in which the transformed equation \eqref{eq:transformed-n-harmonic} is satisfied. Since it is also bounded, it can be used as a test function. 

\subsection{Left hand side estimates}

Let us test the transformed equation \eqref{eq:transformed-n-harmonic} with $\psi = \zeta_0 \phi$. We shall check here that the left hand side integral is bounded from below by $\int_{B(a,r)} |\nabla u|^{n-\eps}$, up to some less important terms.  

The main issue is to show that the divergence-free error term $V$ is in fact small. For this, first note that in the case $\eps=0$ we have Hodge decomposition (Theorem \ref{th.Hodge1})
\[
P^\top \nabla \tu = \nabla w + W,
\]
where $W$ is small by the commutator estimate (Theorem \ref{th:CRW-commutator})
\[
\| W \|_{L^{n-\eps}} \ale \| P^\top \|_\BMO \| \nabla \tu \|_{L^{n-\eps}}.
\]
Thanks to \eqref{eq:P-est}, both norms are bounded in terms of the $(n-\eps)$-energy of $u$: 
\begin{gather*}
\| P^\top \|_{\BMO} 
\ale \| \nabla P^\top \|_{L^{p,n-p}} 
\ale \| \nabla u \|_{L^{n-\eps,\eps}(B(a,3r))} \\
\| \nabla \tilde{u} \|_{L^{n-\eps}} 
\ale \| \nabla u \|_{L^{n-\eps}(B(a,3r))}.
\end{gather*}
In consequence, $W$ is also small when compared to $\nabla u$: 
\[
\| W \|_{L^{n-\eps}} \ale \| \nabla u \|_{L^{n-\eps,\eps}(B(a,3r))} \| \nabla u \|_{L^{n-\eps}(B(a,3r))}.
\]
As in \cite{Iwaniec-Sbordone-1994}, let us denote by $T$ the Calder\'{o}n--Zygmund operator which produces the divergence-free part of the Hodge decomposition, so that 
\[
T(P^\top \nabla \tu) = W, \quad T(|\nabla \tu|^{-\eps} P^\top \nabla \tu) = V.
\]
Following Theorem \ref{thm:is-stability}, we also introduce the nonlinear operator 
\[
S^\eps \colon L^{n-\eps} \to L^{\frac{n-\eps}{1-\eps}}, \qquad 
S^\eps(f) := \left( \frac{|f|}{\| f \|_{n-\eps}} \right)^{-\eps} f.
\]
Note that $\| S^\eps(f) \|_{L^{\frac{n-\eps}{1-\eps}}} = \| f \|_{L^{n-\eps}}$ thanks to the normalization. The main ingredient in bounding $V$ is provided by the following stability estimate, cf. \cite[Th.~4]{Iwaniec-Sbordone-1994} and Theorem~\ref{thm:is-stability} in Section 2,
\[
\| T S^\eps(f) - S^\eps(Tf) \|_{L^{\frac{n-\eps}{1-\eps}}} \le C \eps \cdot \| f \|_{L^{n-\eps}} 
\quad \text{for } f \in L^{n-\eps}.
\]
In our case, as $P^\top$ is orthogonal, we have 
\[
S^\eps(P^\top \nabla \tu) = \| \nabla \tu \|_{L^{n-\eps}}^\eps \cdot |\nabla \tu|^{-\eps} P^\top \nabla \tu, \qquad 
T S^\eps(P^\top \nabla \tu) = \| \nabla \tu \|_{L^{n-\eps}}^\eps \cdot V.
\]
On the other hand, $T(P^\top \nabla \tilde{u}) = W$ and in consequence 
\[
\| S^\eps (T(P^\top \nabla \tu)) \|_{L^{\frac{n-\eps}{1-\eps}}} 
= \| T(P^\top \nabla \tu) \|_{L^{n-\eps}} = \| W \|_{L^{n-\eps}}.
\]
By triangle inequality, we have now 
\begin{align*}
\| \nabla \tu \|_{n-\eps}^\eps \cdot \| V \|_{L^{\frac{n-\eps}{1-\eps}}} 
& = \| T S^\eps(P^\top \nabla \tu)\|_{L^{\frac{n-\eps}{1-\eps}}} \\
& \le \| S^\eps(T(P^\top \nabla \tu)) \|_{L^{\frac{n-\eps}{1-\eps}}} 
+ \| T S^\eps(P^\top \nabla \tu) - S^\eps(T(P^\top \nabla \tu)) \|_{L^{\frac{n-\eps}{1-\eps}}} \\
& \le \| W \|_{L^{n-\eps}} + C \eps \cdot \| P^\top \nabla \tu \|_{L^{n-\eps}}\, .
%\\
%& \le \| W \|_{n-\eps} + C \eps \cdot \| \nabla u \|_{n-\eps}.
\end{align*}
Given our estimates for $W$, we thus have
\begin{equation}
\label{V-small}
\| V \|_{L^{\frac{n-\eps}{1-\eps}}} 
\ale \left( \eps + \| \nabla u \|_{L^{n-\eps,\eps}(B(a,3r))} \right) 
\cdot \| \nabla u \|_{L^{n-\eps}(B(a,3r))}^{1-\eps}.
\end{equation}

\bigskip

We are now ready to investigate the left hand side. Thanks to the decomposition $\nabla \phi = |\nabla u|^{-\eps} P^\top \nabla u - V$, it consists of three terms:
\begin{align*}
- \int \dv(|\nabla u|^{n-2} P^\top \nabla u) \cdot (\zeta_0 \phi) 
& = \int |\nabla u|^{n-2} \ps{ P^\top \nabla u }{\nabla (\zeta_0 \phi)} \\
& = \int \zeta_0 |\nabla u|^{n-2} \ps{ P^\top \nabla u }{|P^\top \nabla u|^{-\eps} P^\top \nabla u} 
- \int \zeta_0 |\nabla u|^{n-2} \ps{ P^\top \nabla u }{V} \\
& \phantom{=} + \int |\nabla u|^{n-2} \ps{ P^\top \nabla u }{\phi \nabla \zeta_0} \\
& = \barroman{I} + \barroman{II} + \barroman{III}.
\end{align*}
Since $\tu = u$ on the support of $\zeta_0$, the first term is simply 
\[
\barroman{I} 
= \int \zeta_0 |\nabla u|^{n-2} |P^\top \nabla u|^{2-\eps}
= \int \zeta_0 |\nabla u|^{n-\eps} 
\ge \int_{B(a,r)} |\nabla u|^{n-\eps}.
\]

For the second term, we use the estimate on $V$ we derived before: 
\begin{align*}
|\barroman{II}|
& \le \int_{B(a,3r)} |\nabla u|^{n-1} |V| \\
& \le \| \nabla u \|_{L^{n-\eps}(B(a,3r))}^{n-1} 
\cdot \| V \|_{L^{\frac{n-\eps}{1-\eps}}(B(a,3r))} \\
& \ale \| \nabla u \|_{L^{n-\eps}(B(a,3r))}^{n-1} 
\cdot \left( \eps + \| \nabla u \|_{L^{n-\eps,\eps}(B(a,3r))} \right) 
\cdot \| \nabla u \|_{L^{n-\eps}(B(a,3r))}^{1-\eps} \\
& = \left( \eps + \| \nabla u \|_{L^{n-\eps,\eps}(B(a,3r))} \right) 
\cdot \int_{B(a,3r)} |\nabla u|^{n-\eps}.
\end{align*}
The implicit constant $C$ above is independent of $\eps$, hence from now on we assume that $\eps$ is chosen small enough so that $C \eps \le \frac{1}{20}$. Moreover, we may assume that the ball $B(a,3r)$ is chosen sufficiently small, and hence $\| \nabla u \|_{L^n(B(a,3r))}$ is also small. In consequence, by H\"older's inequality we may have $C \| \nabla u \|_{L^{n-\eps,\eps}(B(a,3r))}$ smaller than another $\frac{1}{20}$. Finally, this means that 
\[
|\barroman{II}|
\le \frac{1}{10} \int_{B(a,3r)} |\nabla u|^{n-\eps}.
\]

The cut-off term $\barroman{III}$ is treated in a standard way, by applying Poincar\'e's inequality to $\phi$. Note that $\nabla \zeta_0$ is non-zero only on the annulus $A := B(a,2r) \setminus B(a,r)$, and hence 
\begin{align*}
|\barroman{III}|
& \ale r^{-1} \int_A |\nabla u|^{n-1} |\phi| \\
& \le r^{-1} \| \nabla u \|_{L^{n-\eps}(A)}^{n-1} \| \phi \|_{L^{\frac{n-\eps}{1-\eps}}(B(a,3r))} \\
& \ale \| \nabla u \|_{L^{n-\eps}(A)}^{n-1} \| 
\nabla \phi \|_{L^{\frac{n-\eps}{1-\eps}}(B(a,3r))} \\
& \ale \| \nabla u \|_{L^{n-\eps}(A)}^{n-1} \| \nabla u \|_{L^{n-\eps}(B(a,3r))}^{1-\eps}
\end{align*}
Finally, by Young's inequality, 
\[
|\barroman{III}|
\le \frac{1}{10} \int_{B(a,3r)} |\nabla u|^{n-\eps} + C(n) \int_A |\nabla u|^{n-\eps}.
\]

\medskip

Summing up the contributions of $\barroman{I},\barroman{II},\barroman{III}$, we finally obtain the left-hand side estimate 
\begin{align}
\label{eq:n-harm-LHS}
\int |\nabla u|^{n-2} \ps{ P^\top \nabla u }{\nabla (\zeta_0 \phi)}
& \ge \int_{B(a,r)} |\nabla u|^{n-\eps}\, dx \\
\nonumber & \quad{} - C(n) \int_{A}|\nabla u|^{n-\varepsilon}\, dx - \frac{1}{5} \int_{B(a,3r)}|\nabla u|^{n-\varepsilon}\, dx\, .
\end{align}

\subsection{Right hand side estimates}

The crucial insight we will employ is that the right hand side of \eqref{eq:transformed-n-harmonic} is small in $L^1$. Note the naive estimate 
\[
\| \wOmega \|_{L^n(\R^n)} 
\ale \| \Omega \|_{L^n(\R^n)} 
\ale \| \nabla u \|_{L^n(B(a,3r))},
\]
which leads to bounding the right hand side $|\nabla \tu|^{n-2} \wOmega \cdot P^\top \nabla \tu$ in $L^1$ by $\| \nabla u \|_{L^n(B(a,3r))}^n$. However, Uhlenbeck's decomposition allows an essential improvement here: 

\begin{lemma}
\label{lem:L1-smallness}
The $L^n$ norm of the matrix $\wOmega$ is bounded by 
\[
\| \wOmega \|_{L^n(\R^n)} \ale 
\| \nabla u \|_{L^{n-\eps,\eps}(B(a,3r))}
\cdot \| \nabla \tu \|_{L^n(\R^n)}
\]
\end{lemma}

\begin{proof}
Let $T \colon L^n(\R^n) \to L^n(\R^n)$ be the operator responsible for the divergence-free part of Hodge decomposition. That way, $\wOmega = T \wOmega$. Recalling that $\wOmega = P^\top \Omega P + P^\top \nabla P$, we have 
\[
\| \wOmega \|_{L^n} 
= \| T \wOmega \|_{L^n} 
\le \| T (P^\top \Omega P) \|_{L^n} + \| T(P^\top \nabla P) \|_{L^n}.
\]

\medskip

Let us first look at the second term. Recall that $P$ was extended as a constant function outside of $B$, so $T(\nabla P) = 0$. By the commutator estimate (Theorem \ref{th:CRW-commutator}), 
\[
\| T(P^\top \nabla P) \|_{L^n}
\ale \| P^\top \|_{\BMO} \| \nabla P \|_{L^n}.
\]
These norms are bounded using \eqref{eq:P-est} and Poincar\'e's inequality:
\begin{align*}
\| \nabla P \|_{L^n}
& \ale \| \nabla \tu \|_{L^n(\R^n)}, \\
\| P^\top \|_{\BMO} 
& \ale \| \nabla P^\top \|_{L^{p,n-p}} 
\ale \| \nabla u \|_{L^{n-\eps,\eps}(B(a,3r))}, 
\end{align*}
which leads to the desired estimate. 

\medskip

The term $\| T (P^\top \Omega P) \|_{L^n}$ behaves similarly, once we unfold the definition of $\Omega$: 
\[
(P^\top \Omega P)_{ij} 
= \sum_{s,t} P^\top_{is} \Omega_{st} P_{tj} 
= \sum_{s,t,l} P^\top_{is} P_{tj} (A^s_{tl}(\tu) - A^t_{sl}(\tu)) \nabla \tu^l.
\]
Once again, there is a gradient field inside, which enables us to use the commutator estimate: 
\[
\| T (P^\top \Omega P) \|_{L^n}
\ale \sum_{s,t,l} \| P^\top_{is} P_{tj} (A^s_{tl}(\tu) - A^t_{sl}(\tu)) \|_{\BMO} \| \nabla \tu^l \|_{L^n}.
\]
For the $\BMO$ estimate, we apply Poincar\'e's inequality once again. Since $P^\top$, $P$ and $A$ are all uniformly bounded, by Leibniz's rule we have 
\begin{align*}
\| P^\top_{is} P_{tj} (A^s_{tl}(\tu) - A^t_{sl}(\tu)) \|_{\BMO}
& \ale \| \nabla (P^\top_{is} P_{tj} (A^s_{tl}(\tu) - A^t_{sl}(\tu))) \|_{L^{p,n-p}} \\
& \ale \| \nabla P^\top_{is} \|_{L^{p,n-p}} + \| \nabla P_{tj} \|_{L^{p,n-p}} + \| \nabla (A^s_{tl}(\tu) - A^t_{sl}(\tu)) \|_{L^{p,n-p}}.
\end{align*}
The estimates for $P^\top$ and $P$ follow as before. The estimate for $A$ is even easier, as $|\nabla (A(\tu))| \ale |\nabla \tu|$ pointwise and so Lemma \ref{lem:BMO-est-NEW} implies the desired bound in terms of $\| \nabla u \|_{L^{p,n-p}(B(a,3r))}$ and hence also $\| \nabla u \|_{L^{n-\eps,\eps}(B(a,3r))}$. 
\end{proof}

\medskip

With Lemma \ref{lem:L1-smallness}, the estimate is straightforward. First, we note the $L^1$ estimate for the right hand side of the equation: 
\[
\| |\nabla u|^{n-2} \wOmega \cdot P^\top \nabla u \|_{L^1(B(a,2r))} 
\le \| \nabla \tu \|_{L^n}^{n-1} \cdot \| \wOmega \|_{L^n}
\ale \| \nabla \tu \|_{L^n}^n \cdot 
\| \nabla u \|_{L^{n-\eps,\eps}(B(a,3r))},
\]
which is an improvement when compared with a naive bound $\| \nabla u \|_{L^n(B(a,3r))}^{n}$. Then we combine it with the $L^\infty$ estimate for the test function $\psi = \zeta_0 \phi$: 
\[
\| \psi \|_{L^\infty(B(a,2r))} 
\le \| \phi \|_{L^\infty(B(a,2r))} 
\ale C(\eps) r^{\eps} \| \nabla u \|_{L^{n-\eps,\eps}(B(a,3r))}^{1-\eps},
\]
finally obtaining 
\[
\left| \int_{B(a,2r)} |\nabla u|^{n-2} \wOmega \cdot P^\top \nabla u \cdot \psi \right|
\ale 
C(\eps) r^{\eps} \cdot \| \nabla \tu \|_{L^n}^n \cdot 
\| \nabla u \|_{L^{n-\eps,\eps}(B(a,3r))}^{2-\eps}.
\]

\medskip

The rest of the reasoning is strictly analogous to the estimate for $H$-systems. Since we aim at a decay estimate for the Morrey norm $\| \nabla u \|_{L^{n-\eps,\eps}}$, the $L^n$ norm is undesirable. We bound it via Lemma \ref{use-wn22} using the $W^{n/2,2}$ assumption: 
\[
\| \nabla \tu \|_{L^n}^n 
\ale \Psi(a,3r)^2 \, \|\nabla u\|_{L^{n-\eps,\eps}(B(a,3r))}^{n-2},
\]
where 
\[
\Psi(a,r) := 
    \| \nabla u \|_{L^n(B(a,r))} + \| \nabla^2 u \|_{L^{n/2}(B(a,r))} + \ldots + \| \nabla^{n/2} u \|_{L^2(B(a,r))}
\]
can be assumed arbitrarily small for small balls. With this estimate, the right-hand side is bounded as follows: 
\begin{equation}
\label{eq:n-harm-RHS}
\left| \int_{B(a,2r)} |\nabla u|^{n-2} \wOmega \cdot P^\top \nabla u \cdot \psi \right|
\ale 
C(\eps) r^{\eps} \cdot \Psi(a,3r)^2 \cdot 
\| \nabla u \|_{L^{n-\eps,\eps}(B(a,3r))}^{n-\eps}.
\end{equation}

\subsection{Choice of $\epsilon$ and hole filling} 
\label{ch:n-harm-finish}

Combining the estimates for both sides of the equation \eqref{eq:n-harm-LHS} with \eqref{eq:n-harm-RHS}, we obtain 
\[
\int_{B(a,r)} |\nabla u|^{n-\eps}\, dx 	\le 
 C(n) \int_{A}|\nabla u|^{n-\eps}\, dx 
 + \frac{1}{5} \int_{B(a,3r)}|\nabla u|^{n-\eps}\, dx
 + C(\eps) r^\eps\, \Psi(a,3r)^2\, 
 \|\nabla u\|_{L^{n-\eps,\eps}(B(a,3r))}^{n-\eps} \, .
\]
The reasoning presented in the case of $H$-systems (Section \ref{sec:hole}) applies here without changes. Assuming we work on a small ball, iterating this estimate leads to H\"older continuity by an application of the Dirichlet growth theorem.

\bigskip

\bibliography{wn2}{}
\bibliographystyle{amsplain}

\vspace*{.5cm}

\end{document}